\documentclass[oneside,english]{amsart}
\usepackage[T1]{fontenc}
\usepackage[latin9]{inputenc}
\usepackage{amstext}
\usepackage{amsthm}
\usepackage{amssymb}
\usepackage{graphicx}
\usepackage{nameref}

\makeatletter

\newcommand{\lyxmathsym}[1]{\ifmmode\begingroup\def\b@ld{bold}
  \text{\ifx\math@version\b@ld\bfseries\fi#1}\endgroup\else#1\fi}

\numberwithin{equation}{section}
\numberwithin{figure}{section}
\theoremstyle{plain}
\newtheorem*{thm*}{\protect\theoremname}
\theoremstyle{plain}
\newtheorem{thm}{\protect\theoremname}
\theoremstyle{definition}
\newtheorem{defn}[thm]{\protect\definitionname}
\theoremstyle{remark}
\newtheorem{rem}[thm]{\protect\remarkname}
\theoremstyle{remark}
\newtheorem{claim}[thm]{\protect\claimname}
\theoremstyle{plain}
\newtheorem{lem}[thm]{\protect\lemmaname}
\theoremstyle{plain}
\newtheorem{prop}[thm]{\protect\propositionname}
\theoremstyle{remark}
\newtheorem*{rem*}{\protect\remarkname}

\makeatother

\usepackage{babel}
\providecommand{\claimname}{Claim}
\providecommand{\definitionname}{Definition}
\providecommand{\lemmaname}{Lemma}
\providecommand{\propositionname}{Proposition}
\providecommand{\remarkname}{Remark}
\providecommand{\theoremname}{Theorem}

\begin{document}
\global\long\def\rr{\mathbb{R}}%
\global\long\def\pp{\mathbb{P}}%
\global\long\def\diam{\text{diam}}%
\global\long\def\eff{\mathcal{F}}%
\global\long\def\capi{\text{Cap}}%
\global\long\def\nn{\mathbb{N}}%
\global\long\def\ee{\mathbb{E}}%
\global\long\def\oo{1}%
\global\long\def\half{\frac{1}{2}}%
\global\long\def\const{\textup{C}}%
\global\long\def\fourth{\frac{1}{4}}%
\global\long\def\zz{\mathbb{Z}}%
\global\long\def\sige{\mathcal{E}}%
\global\long\def\vol{\textup{Vol}}%
\global\long\def\qed{\hfill$\blacksquare$}%

\title[Bounds for the Correlation Length of the 2D RFIM]{Upper and Lower Bounds for the Correlation Length of the Two-Dimensional
Random-Field Ising Model}
\author{Yoav Bar-Nir}
\begin{abstract}
We study the rate of correlation decay in the two-dimensional random-field
Ising model at weak field strength $\varepsilon$. We combine elements
of the recent proof of exponential decay of correlations with a quantitative
refinement of a result of Aizenman--Burchard on the tortuosity of
random curves to obtain an upper bound of the form $\exp(\exp(O(1/\varepsilon^{2})))$
on the correlation length of the model at all temperatures. Conversely,
we show, by adapting methods of Fisher--Fr\"{o}hlich--Spencer, that
on square domains of side length as large as $\exp(O(1/\varepsilon^{2/3}))$
the model continues to exhibit strong dependence on boundary conditions
at low temperature.
\end{abstract}

\maketitle

\section{Overview}

The random-field Ising model (RFIM) is a prime example of a disordered
spin system. It is obtained by subjecting the standard Ising model
to a random (quenched) external magnetic field composed of independent
and identically distributed random variables. The model is described
by the formal random Hamiltonian
\begin{equation}
H^{h}(\sigma):=-J\sum_{u\sim v}\sigma_{u}\sigma_{v}-\sum_{v}(\eta+\varepsilon h_{v})\sigma_{v},\label{hamiltonian rfim}
\end{equation}
where the Ising spins $\sigma$ take values in $\{-1,1\}$, $J>0$
is the coupling strength, $\varepsilon>0$ is the disorder intensity
(or the strength of the random field), $\eta\in\rr$ is the intensity
of the homogeneous external field and $(h_{v})_{v\in\zz^{2}}$ is
the random field, which we take here to consist of independent standard
Gaussian random variables. The seminal work of Imry--Ma \cite{imrymaargument}
predicted that the addition of the random field to the \emph{two-dimensional}
Ising model causes the model to lose its ordered low-temperature phase
and become disordered at all temperatures (including zero temperature),
for every disorder intensity $\varepsilon>0$. This prediction was
given a rigorous proof in the celebrated work of Aizenman--Wehr \cite{aizenman1989roundingPRL,aizenman1990roundingCMP}.
We discuss here the question of quantifying the rate of correlation
decay, as captured by the order parameter
\begin{equation}
m(L):=\frac{1}{2}\left(\ee\left[\left\langle \sigma_{0}\right\rangle _{\Lambda(L)}^{+}\left(h_{v}\right)\right]-\ee\left[\left\langle \sigma_{0}\right\rangle _{\Lambda(L)}^{-}\left(h_{v}\right)\right]\right),\label{eq:order parameter def}
\end{equation}
where $\left\langle \sigma_{0}\right\rangle _{\Lambda(L)}^{+}\left(h_{v}\right)$
and $\left\langle \sigma_{0}\right\rangle _{\Lambda(L)}^{-}\left(h_{v}\right)$
denote the thermal expectation value of the spin at the origin for
the given realization $(h_{v})$ of the random field, when the model
is sampled in the discrete square $\Lambda(L):=\{-L,\ldots,L\}^{2}$
with $+$ or $-$ boundary conditions, respectively, and where the
operator $\ee$ denotes expectation over the values of the random
field $(h_{v})$. The order parameter controls several related notions
of correlation decay, as discussed in \cite[Section 1.2]{aizenman2019power}.

It is generally expected that at high enough disorder, be it thermal
or due to noisy environment, correlations decay exponentially fast.
Results in this vein for systems related to the RFIM can be found
in the works of Berretti \cite{berretti1985some}, Imbrie--Fr\"{o}hlich
\cite{frohlich1984improved} and Camia--Jiang--Newman \cite{camia2018note}.
The main challenge thus lies in analyzing the behavior at low temperature
and weak disorder strength. In recent years there has been major progress
in quantifying the rate of correlation decay for the two-dimensional
RFIM in the latter regime. Upper bounds were established in a series
of works \cite{chatterjee2018decay,aizenman2019power,expdecay,ding2019exponentialzerotemp,ding2019exponentialpostemp}
which culminated in a proof of exponential decay of correlations at
all temperatures and all positive disorder strengths. Precisely, the
following theorem was proved.
\begin{thm*}
[exponential decay of correlations \cite{ding2019exponentialpostemp,expdecay}]
In the nearest-neighbor random-field Ising model on $\zz^{2}$, as
specified by (\ref{hamiltonian rfim}), at any coupling strength $J>0$
and disorder intensity $\varepsilon>0$ there exist constants $C(J/\varepsilon),c(J/\varepsilon)>0$,
depending only on the ratio $J/\varepsilon$, such that at all temperatures
$T\geq0$ and homogeneous external field $\eta\in\rr$ the order parameter
satisfies for all integer $L\ge1$,
\begin{equation}
m(L)\leq C(J/\varepsilon)\exp(-c(J/\varepsilon)L).\label{eq:exp decay-1}
\end{equation}
\end{thm*}
While the theorem establishes exponential decay of correlations, it
leaves open the question of how the correlation length varies as the
ratio $J/\varepsilon$ tends to infinity (at low temperature). Here,
the notion of ``correlation length'' can be given several interpretations:
One standard definition is the infimum over $\zeta$ for which $m(L)\le e^{-L/\zeta}$
for all sufficiently large $L$; denote this value by $\zeta_{1}=\zeta_{1}(T,J,\eta,\varepsilon)$.
A second possibility is the minimal value of $L$ for which $m(L)$
drops below some fixed threshold $m$ (e.g., $m=1/2$); denote this
value by $\zeta_{2}:=\zeta_{2}(T,J,\eta,\varepsilon,m)$. In \cite{expdecay}
it was asked to determine the order of magnitude of the correlation
length. It was noted that $\zeta_{2}\le\exp\left(\exp\left(O\left((\frac{J}{\varepsilon})^{2}\right)\right)\right)$
was established in \cite{aizenman2019power} (for each fixed $T,\eta$
and $m$) and that the behavior $\zeta_{1}=\exp(O((J/\varepsilon)^{2}))$
was discussed in \cite{bricmont1988hierarchical}.

The goal of this work is to provide upper and lower bounds on the
correlation length of the two-dimensional Ising model. The following
is our main result.
\begin{thm}
\label{thm:clength upper-1} Consider the nearest-neighbor random-field
Ising model on $\zz^{2}$, as specified by (\ref{hamiltonian rfim}),
at coupling strength $J>0$ and disorder intensity $\varepsilon>0$
satisfying that $\frac{J}{\varepsilon}\ge1$.
\begin{enumerate}
\item There exists a universal constant $C>0$ such that at all temperatures
$T\ge0$ and homogeneous external field $\eta\in\rr$, the correlation
length $\zeta_{1}$ satisfies
\begin{equation}
\zeta_{1}(T,J,\eta,\varepsilon)\leq\exp\left(\exp\left(C\left(\frac{J}{\varepsilon}\right)^{2}\right)\right).\label{eq:corr length bound upper-1}
\end{equation}
\item For each $0<\delta<1$ there exists $c(\delta)>0$ (depending only
on $\delta$) such that at zero temperature and zero homogeneous external
field the correlation length $\zeta_{2}$ satisfies
\begin{equation}
\zeta_{2}(0,J,0,\varepsilon,1-\delta)\ge\exp\left(c(\delta)\left(\frac{J}{\varepsilon}\right)^{2/3}\right).\label{eq:lower bound}
\end{equation}
In other words, when $L<\exp\left(c(\delta)(\frac{J}{\varepsilon})^{2/3}\right)$
the spin at the origin of the ground state in $\Lambda(L)$ with $+$
boundary conditions is equal to $1$ with probability (over the random
field) greater than $1-\frac{1}{2}\delta$.
\end{enumerate}
\end{thm}

Our approach to the correlation length upper bound (\ref{eq:corr length bound upper-1})
builds upon the recent proofs of exponential decay of correlations
\cite{ding2019exponentialpostemp,expdecay}. While the available proofs
do not provide an explicit upper bound on the correlation length,
it was noted in \cite{expdecay} that such a bound will follow from
a quantitative refinement of one of the main tools of the proof, the
Aizenman--Burchard theorem on the tortuosity of random curves \cite{aizemanburchard}.
In Section \ref{sec:The-Aizenman=002013Burchard-Theorem} we provide
such a refinement, which is then used in Section \ref{sec:Upper-Bound-in}
to derive the upper bound (\ref{eq:corr length bound upper-1}). Our
quantitative refinement of the Aizenman--Burchard theorem, given
in Theorem \ref{thm:exact roughness}, may be of use in other contexts
as well.

Section \ref{sec:Lower-Bound} and Appendix \ref{sec:Appendix:-Proving-the}
are devoted to the proof of the lower bound (\ref{eq:lower bound}),
derived in a somewhat more general setting. The proof adapts to the
two-dimensional setting the ``coarse graining'' methods of Fisher--Fr\"{o}hlich--Spencer
\cite{fisher} which were developed in their discussion of the phase
transition that the three-dimensional RFIM displays at low temperature
as the random-field strength is varied (the transition was given rigorous
proofs in the celebrated works of Imbrie \cite{imbrie1985ground}
and Bricmont--Kupiainen \cite{bricmont1988phase}). This approach
will in fact give us a stronger result, which is that below the lower
bound we should expect all spins in $\Lambda(L)$, not just at the
origin, to be equal to 1 with high probability.

While this work was in progress, a sharper estimate of the correlation
length $\zeta_{2}$ was established by Ding--Wirth \cite{ding2020correlation}.
For $\frac{J}{\varepsilon}\ge1$ and $0<\delta<1$, they proved the
lower bound $\zeta_{2}(0,J,0,\varepsilon,\delta)\ge\exp\left(c(\delta)(\frac{J}{\varepsilon})^{4/3}\frac{1}{\log(J/\varepsilon)}\right)$
at zero temperature and the near-matching upper bound $\zeta_{2}(T,J,0,\varepsilon,\delta)\le\exp\left(C(\delta)(\frac{J}{\varepsilon})^{4/3}\right)$
at all temperatures $T\ge0$.

\section{\label{sec:The-Aizenman=002013Burchard-Theorem}The Aizenman--Burchard
Theorem and a Quantitative Refinement}

\subsection{A Brief Introduction}

In this section we present a quantitative refinement to a result by
Aizenman and Burchard regarding fractality of random curves in $\rr^{d}$
\cite{aizemanburchard}.

A system of random curves is a collection of set-valued random variables,
$(\mathcal{F}_{\delta})_{\delta>0}$, where each element of $\mathcal{F}_{\delta}$
is some piecewise linear curve, composed of line segments of length
$\delta$. In essence, Aizenman and Burchard showed that if a system
of random curves satisfies an assumption, which we will call \textbf{H2},
then as $\delta$ gets lower, $\eff_{\delta}$ will resemble a collection
of curves of Hausdorff dimension greater than some constant $d_{\min}>1$
independent of $\delta$. Our goal in this section is to quantify
both $d_{\min}$ and also the rate at which $\eff_{\delta}$ starts
to resemble such a collection of curves. 

The assumption \textbf{H2 }depends on three parameters, $0<\rho<1,\sigma>0$,
and $K>0$, and it goes roughly as follows; for any set of $n$ cylinders
in $\rr^{d}$ with aspect ratio $\sigma$ and length greater than
$\delta$, the probability that each of the cylinders is crossed by
some curve in $\mathcal{F}_{\delta}$ is less than $K\rho^{n}$. Our
improvement to the Aizenman--Burchard theorem will be to quantify
the fractality of a system of curves $(\mathcal{F}_{\delta})_{\delta>0}$
satisfying \textbf{H2 }in terms of the parameter $\rho$ when it's
close to 1. More specifically, assuming $\rho=1-\varepsilon$ for
small enough $\varepsilon>0$, we will show the lower bound
\[
d_{\min}\geq1+\kappa\frac{\varepsilon^{2}}{\log\left(1/\varepsilon\right)^{3}},
\]
for a universal constant $\kappa$ independent of $\rho$.

This quantification, in addition to a more technical one we will define
later, will be crucial later in this paper for obtaining the upper
bound (\ref{eq:corr length bound upper-1}).

\subsection{Definitions\label{subsec:Definitions}}
\begin{defn}
\label{def: well seperated}A finite collection of subsets in $\rr^{d}$
is called \textbf{well separated }if for any two sets $A,B$ in it:
\[
d(A,B)\geq\max(\diam(A),\diam(B)),
\]
where $d(A,B):=\inf\{d(x,y)\,|\,x\in A,\,y\in B\}$, $d(x,y)$ is
the Euclidean distance between $x,y\in\mathbb{R}^{d}$, and where
$\diam(A):=\sup\{d(x,y)\,|\,x,y\in A\}$ is the Euclidean diameter
of $A$.
\end{defn}

\begin{defn}
\label{def:polygonal path}A \textbf{$\delta$-polygonal path} in
$\rr^{d}$ is a piece-wise linear continuous path, $\gamma:I\rightarrow\rr^{d}$
(where $I\subseteq\rr$ is some closed interval), formed by the concatenation
of finitely many segments of length $\delta$.
\end{defn}

\begin{defn}
\label{def:randcurve system}A \textbf{system of random curves with
short variable cutoff in a compact set $\Lambda\subseteq\rr^{d}$
}is a collection of set-valued random variables $\mathcal{F}=\left(\mathcal{F_{\delta}}\right)_{0<\delta\leq1}$,
with each $\mathcal{\eff_{\delta}}$ equaling a random finite set
of $\delta$-polygonal paths. Such collections form closed sets in
the Hausdorff metric contained in $\Lambda$ and the sigma algebra
in this case is the one induced by the Hausdorff metric. We will usually
denote the probability measure of the space on which $\mathcal{F_{\delta}}$
is defined by $\pp_{\delta}$.
\end{defn}

In a way, a system of random curves with short variable cutoff is
a substitute for sampling a random \emph{continuous curve} (or collection
of curves) via discrete approximations. We sample random piece-wise
linear curves with step length $\delta$, and the smaller $\delta$
gets the more complex our curves become.
\begin{defn}
\label{def:h2}A system of random curves with short variable cutoff
$\mathcal{F}=\left(\mathcal{F_{\delta}}\right)_{0<\delta\leq1}$ in
some compact $\Lambda\subseteq\rr^{d}$ is said to satisfy \textbf{hypothesis
H2} with parameters $0<\rho<1$ and$\,\sigma\geq1$, if there exists
a constant $C>0$ such that for any $\delta>0$ and any finite collection
of well-separated cylinders $(A_{i})_{i=1}^{m}\subseteq\Lambda$ with
aspect ratio $\sigma$ and lengths $>\delta$ it holds that
\begin{equation}
\pp_{\delta}\left(\textup{\ensuremath{\bigcap_{i=1}^{m}\{}\ensuremath{A_{i}} is crossed in the long direction by a curve in \ensuremath{\mathcal{F}_{\delta}}\}}\right)\leq C\rho^{m},\label{eq:H2}
\end{equation}
and the constant $C$ is independent of $m$ and the chosen cylinders.
Here the length of the cylinder is the length between its two circular
bases, and the aspect ratio $\sigma$ is the ratio between its length
and the radius of the bases. By a crossing of a cylinder we mean a
path contained in the cylinder which connects its two bases. Note
that the bases need not be parallel to one of the coordinate hyper-planes
$\{x_{i}=0\}$.
\end{defn}

\begin{rem}
\label{rem:In-the-original}In the original theorem \cite{aizemanburchard},
the assumption that $\sigma\geq1$ is ommited and in its place a crossing
of the cylinder is defined as a crossing of its long side rather than
a crossing from one base to the other. Here we will simplify things
by just looking at the $\sigma\geq1$ and only at crossings from one
base to the other. As will be seen later, our proof will also work
if the assumption of $\sigma\geq1$ is replaced by $\sigma\geq\sigma_{0}$
for some $\sigma_{0}>0$.

\begin{figure}
\begin{centering}
\includegraphics[scale=0.3]{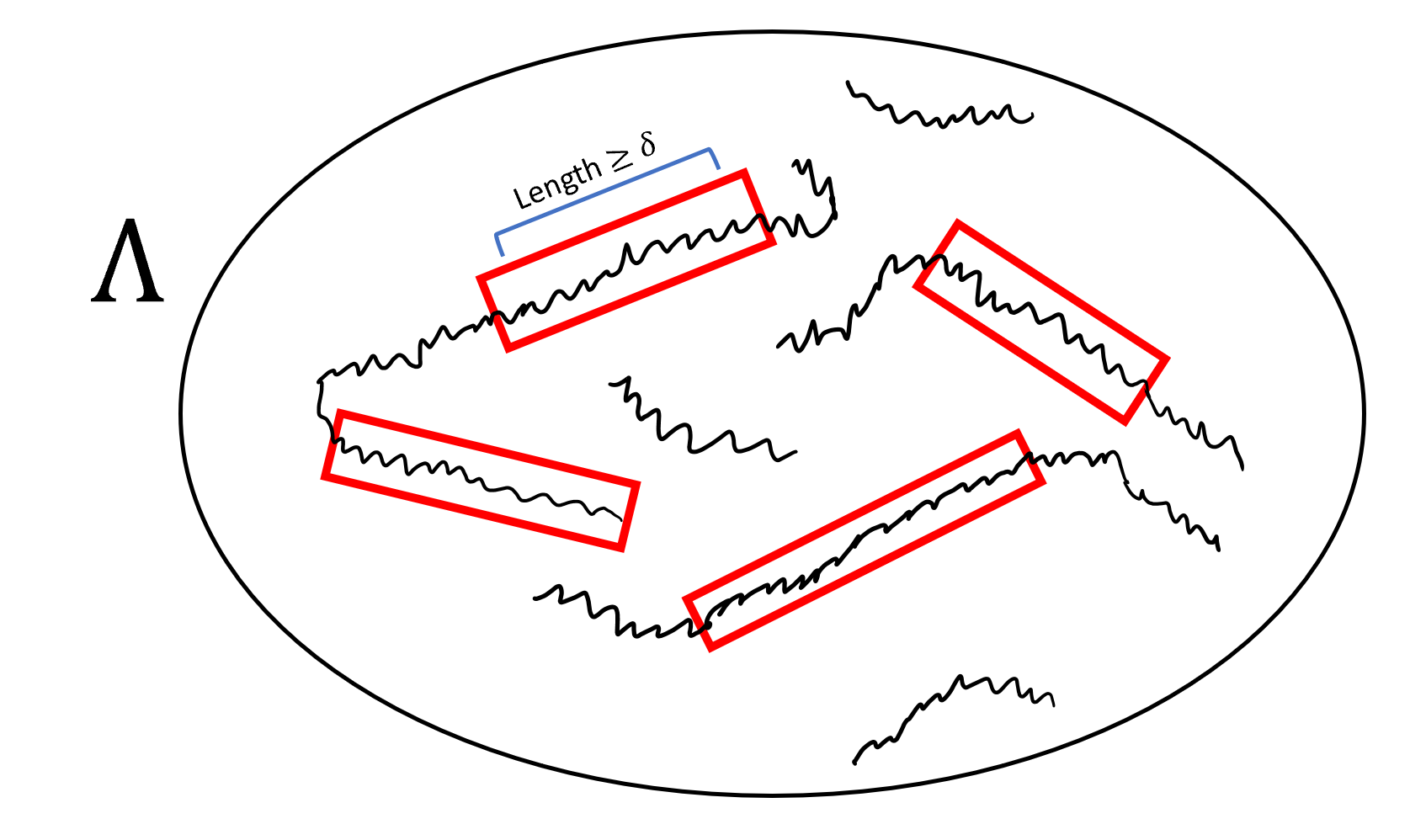}
\par\end{centering}
\caption{In order for hypothesis H2 to hold, we need that for every collection
of cylinders (with sufficient length and separation) the probability
of all of them being crossed by curves in $\mathcal{F_{\delta}}$
is exponentially decreasing with the number of cylinders.}
\end{figure}
\end{rem}

\begin{defn}
For a set $A\subseteq\rr^{d}$, parameters $\ell>0$ and $s$, we
define the $s,\ell$ capacity of $A$ to be:
\begin{equation}
\frac{1}{\capi_{s;\ell}(A)}=\inf_{\mu\geq0,\mu(A)=1}\iint_{A\times A}\frac{\mu(dx)\mu(dy)}{\max(\left|x-y\right|,\ell)^{s}},\label{eq: cap def}
\end{equation}
where the infimum is over all probability measures on $A$.

The following properties of the capacity from \cite{aizemanburchard}
will be of use to us:
\end{defn}

\begin{claim}
\label{claim:cap properties}Let $A\subset\rr^{d}$. Then:

(i) For any covering of $A$ by sets $(B_{j})$ of diameter at least
$\ell$:
\begin{equation}
\sum_{j}\diam(B_{j})^{s}\geq\capi_{s;\ell}(A).\label{eq:capprop1}
\end{equation}

(ii) The minimal number of elements in a covering of $A$ by sets
of diameter $\ell$, $\mathcal{N}(A,\ell)$, satisfies:
\begin{equation}
\mathcal{N}(A,\ell)\geq\capi_{s;\ell}(A)\cdot\ell^{-s}.\label{eq:capprop2}
\end{equation}

(iii) If $\inf_{\ell>0}\capi_{s;\ell}(A)>0$, then the Hausdorff dimension
of $A$ is at least $s$.
\end{claim}

\subsection{The Main Theorem}

Aizenman and Burchard proved the following theorem relating the hypothesis
\textbf{H2 }and the capacity (and thus Hausdorff dimension and tortuosity)
of a random curve system:
\begin{thm}
\label{thm:roughness}Let $\mathcal{F}=(\mathcal{F_{\delta}})_{0<\delta\leq1}$
be a system of random curves with short variable cutoff in $\Lambda\subseteq\rr^{d}$
satisfying hypothesis \textbf{H2}, then there exists some $d_{\min}>1$
such that for any fixed $r>0$ and $s>d_{\min}$, the random variables:
\begin{equation}
T_{s,r;\delta}:=\inf_{\mathcal{C}\in\mathcal{F}_{\delta},\diam(C)\geq r}\capi_{s;\delta}(\mathcal{C})\label{eq:T var}
\end{equation}
satisfy that for every $\varepsilon>0$ there exists a $\mu>0$ such
that for all $\delta$:
\[
\pp(T_{s,r;\delta}\leq\mu)<\varepsilon.
\]
\end{thm}

As stated before, the intuition of the theorem is that as we take
$\delta$ to be smaller and smaller, the random curves will resemble
a set of Hausdorff dimension $>d_{min}$ more and more. This intuition
can be formalized in terms of a scaling limit as in \cite{aizemanburchard},
however we will not use nor need it here. Our goal in this section
will be to give an improvement of the result, by giving quantitative
bounds for the minimal dimension $d_{min}$ and in addition probabilistic
bounds on how the random variables $T_{s,r;\delta}$ are bounded away
from zero. Both of these will be given in terms of the parameter $\rho$
from hypothesis H2\textbf{. }We shall prove the following:
\begin{thm}
\label{thm:exact roughness} Let $\mathcal{F}=(\mathcal{\mathcal{F}_{\delta}})_{0<\delta\leq1}$
be a system of random curves in a compact subset $\Lambda\subset\rr^{d}$
with a short variable cutoff. Suppose $\mathcal{F}$ satisfies hypothesis
\textbf{H2} with parameters $\rho=1-\varepsilon$, where $0<\varepsilon<\frac{1}{10}$,
and $\sigma>0$. Then there exists some constant $\alpha$ satisfying
\begin{equation}
\alpha\geq\kappa\frac{\varepsilon^{2}}{\log\left(1/\varepsilon\right)^{3}},\label{eq:alphabound1}
\end{equation}
so that for all $\delta>0$ and for all $\mu>0$
\begin{equation}
\pp_{\delta}(T_{1+\alpha,1;\delta}<\mu)\leq\mu^{\frac{K}{\log(1/\varepsilon)}},\label{eq:gbound1}
\end{equation}
with $\kappa,K>0$ depending only on $\sigma$ and $\Lambda$.
\end{thm}

\subsection{Straight Runs}

\textbf{Policy regarding constants: }During proofs in this section,
we will use $K$ (or $C$) and $\kappa$ (or $c$) to denote universal
constants depending only on $\sigma$ and $\Lambda$, which may increase
or decrease respectively from line to line.

Our goal in this section will be to give a lower bound for the capacity
in terms of a function on a random variable, $k_{0;s}$.
\begin{defn}
\label{def:straight run}Let $\gamma>1$ be some constant (called
the ``scaling factor''), $\mathcal{C}$ a curve in $\rr^{d}$. A\emph{
$\gamma$-straight run} at scale $L$ for $\mathcal{C}$ is a crossing
of $\mathcal{C}$ from one spherical ``face'' to the other of some
cylinder of length $L$ and radius $\frac{9}{2\sqrt{\gamma}}L$. We
say a straight run is \emph{nested} in another straight run if the
cylinder of the first is contained in the cylinder of the second one. 
\end{defn}

\begin{defn}
\label{def:k, gamma sparse}For a curve $\mathcal{C}$ in $\rr^{d}$,
positive integer $k_{0}$, scaling factor $\gamma>1$, and length
$\ell>0$, we say straight runs in $\mathcal{C}$ are $(\gamma,k_{0})$-sparse
down to length $\ell>0$, if for any $n\geq\frac{1}{2}k_{0}$ and
positive integers $1\leq k_{1}<k_{2}<...<k_{n}\leq2n$, $\gamma^{-k_{n}}\geq\ell$,
there is no nested sequence of $\gamma$-straight runs at scales $\gamma^{-k_{1}},\gamma^{-k_{2}},...,\gamma^{-k_{n}}$.
\end{defn}

Aizenman and Burchard proved the following theorem relating straight
runs and their sparsity to the Hausdorff dimension (\cite{aizemanburchard},
Theorem 5.1)
\begin{thm}
\label{thm:haus dim st runs} Let $\gamma>1$ be a scaling factor
and let $m=\left\lfloor \gamma\right\rfloor $. If in some curve $\mathcal{C}$
in $\rr^{d}$ straight runs are $(\gamma,k_{0})$-sparse (down to
all scales), then:
\begin{equation}
\dim_{\mathcal{H}}(\mathcal{C})\geq\frac{\log\left(m(m+1)\right)}{2\log(\gamma)},\label{eq:haus dim bound}
\end{equation}
where $\dim_{\mathcal{H}}(\mathcal{C})$ denotes the Hausdorff dimension
of $\mathcal{C}$.
\end{thm}

Unfortunately for us, this theorem is only of use in a scaling limit
setting, not the piece-wise linear curves we are interested in. Luckily
though, the above theorem is actually implied from a more general
fact, also proven in \cite{aizemanburchard} (Lemma 5.2 and Lemma
5.4, equation (5.22)).
\begin{lem}
\label{lem:actually useful thing} Let $\mathcal{C}$ be a curve in
$\rr^{d}$ in which straight runs are $(\gamma,k_{0})$-sparse down
to scale $\ell$. Then:
\begin{equation}
\capi_{s;\ell}(\mathcal{C})\geq\left(\left(\frac{\gamma}{m}-1\right)\cdot\diam(\mathcal{C})\right)^{s}\left(\gamma^{s\cdot k_{0}}+\frac{\beta}{1-\gamma^{s}/\beta}\right)^{-1},\label{eq:straight runs inequality}
\end{equation}
where $m$ is any integer in $[\gamma/2,\gamma]$, $\beta=\sqrt{m(m+1)}$,
and $\gamma^{s}<\beta$.
\end{lem}

Note that what matters to us most is the smallest $k_{0}$ for which
straight runs are $(\gamma,k_{0})$-sparse. Our goal in this section
will be to prove a relation between hypothesis \textbf{H2 }and the
distribution of this smallest $k_{0}$. 
\begin{lem}
\label{lem:prob bound on existence of straigh runs} Let $\mathcal{F}=(\mathcal{\mathcal{F}_{\delta}})_{0<\delta\leq1}$
be a system of random curves with short variable cutoff in $\Lambda\subseteq\rr^{d}$
satisfying hypothesis \textbf{H2 }with parameters $\rho=1-\varepsilon,\sigma$.
Suppose $\gamma>\max(4d,\sigma^{2})$. Then for small enough $\varepsilon$
there exists constants $K_{0},K_{1}>0$ depending only on $\sigma$
and $\Lambda$ for which:
\begin{equation}
\pp_{\delta}\left(\begin{array}{c}
\textup{in some curve in \ensuremath{\mathcal{F}_{\delta},} there is a nested sequence of straight runs}\\
\textup{at scales }\gamma^{-k_{1}},\gamma^{-k_{2}},...,\gamma^{-k_{n}}
\end{array}\right)\label{eq:straightrunboundprob}
\end{equation}

\[
\leq K_{1}\gamma^{2dk_{n}}e^{K_{1}n-K_{0}\sqrt{\gamma}\varepsilon n},
\]
for any increasing sequence of positive integers $k_{1},\ldots,k_{n}$
such that $\gamma^{-k_{n}}\geq\delta$.
\end{lem}

\begin{proof}
First note that if a curve in $\rr^{d}$ crosses a cylinder of length
$L$ and radius $\frac{9}{2\sqrt{\gamma}}L$, then it also crosses
a cylinder of radius $\frac{5}{\sqrt{\gamma}}L$ and length $\frac{L}{2}$,
which is centered on a line between two points in the discrete lattice
$\frac{L}{\gamma}\zz^{d}$. Therefore, we can instead bound the probability
of there existing a sequence of straight runs whose cylinders are
centered on a segment in $\gamma^{-k_{1}-1}\zz^{d},\ldots,\gamma^{-k_{n}-1}\zz^{d}$. 

The number of possible positions for the cylinder at scale $\gamma^{-k_{1}}$
will be at most $K_{\Lambda}\gamma^{2dk_{1}}$, where $K_{\Lambda}\approx\vol(\Lambda)$
is a constant only depending on $\Lambda$. Then the number of possible
positions for the second cylinder at scale $\gamma^{-k_{2}}$ will
be at most $K_{d}\gamma^{2d(k_{2}-k_{1})}$ for a constant $K_{d}$
depending only on the dimension. Indeed, to pick the second cylinder
given the first cylinder we need to pick two points in the lattice
$\gamma^{-k_{2}-1}\zz^{d}$ contained in a subset of $\rr^{d}$ with
volume $K_{d}\gamma^{dk_{1}}$ for appropriate constants depending
solely on the dimension. Repeating this until $k_{n}$, we get that
the number of possible cylinders at scales $\gamma^{-k_{1}},\gamma^{-k_{2}},...,\gamma^{-k_{n}}$
is bounded above by:
\begin{equation}
K_{\Lambda}\gamma^{2dk_{1}}K_{d}\gamma^{2d(k_{2}-k_{1})}\cdots K_{d}\gamma^{2d(k_{n}-k_{n-1})}=K_{d}^{n}\gamma^{2dk_{n}}.\label{eq:numberofcylbound}
\end{equation}

For a fixed collection of cylinders $A_{1},..,A_{n}$, with radius
and length as above, we want to bound the probability of all of them
being crossed at once. This will be done using the assumption regarding
hypothesis \textbf{H2. }To match the aspect ratio $\sigma$ in \textbf{H2},
cut each cylinder $A_{i}$ into $\sqrt{\gamma}/10\sigma$ smaller
cylinders of aspect ratio $\sigma$ (whose spherical bases are translates
of the bases of $A_{i}$). If the length of the original cylinder
$A_{i}$ is $L/2$, as described above its radius will be $\frac{5}{\sqrt{\gamma}}L$.
Therefore the length of each cylinder obtained from the cutting is
$\frac{5L}{\sqrt{\gamma}/\sigma}$ and so each has diameter (diameter
as defined in \ref{def: well seperated}, not diameter of the base
of the cylinder) of $L\sqrt{\frac{10^{2}}{\gamma}+\frac{(5\sigma)^{2}}{\gamma}}=5L\sqrt{\frac{4+\sigma^{2}}{\gamma}}$.
To obtain a collection of well seperated cylinders from the ones cut
from $A_{i}$, we pick every second cylinder (since $\sigma\geq1$),
which yields at least $\sqrt{\gamma}/20\sigma$ cylinders. 

Finally, note that each section of $A_{i+1}$ intersects at most two
of the smaller cylinders $A_{i}$ was cut into, so by removing those
two cylinders from each layer $A_{i}$ we get a well seperated collection
with at least $n(\sqrt{\gamma}/20\sigma-2)$. Therefore, using hypothesis
\textbf{H2}:
\begin{equation}
\pp_{\delta}(\textup{All the cylinders \ensuremath{A_{1},..,A_{n}} are crossed)\ensuremath{\leq C(1-\varepsilon)^{n(\sqrt{\gamma}/20\sigma-2)}\leq Ce^{C\sqrt{\gamma}\log(1-\varepsilon)n}\leq Ce^{-K_{0}\sqrt{\gamma}\varepsilon n}}.}\label{eq:crossing bound for fixed cylinders}
\end{equation}
Finally, combining (\ref{eq:numberofcylbound}) and (\ref{eq:crossing bound for fixed cylinders})
and applying a union bound, we get our desired result.
\end{proof}
\begin{rem}
\label{rem:sigma fail}This is the only place where the assumption
$\sigma\geq1$ was needed. Were we to replace this assumption with
$\sigma\geq\sigma_{0}$ instead, we would have had to pick instead
every $\left\lceil \frac{4+\sigma_{0}^{2}}{\sigma_{0}}\right\rceil $-th
layer instead. This wouldv'e yielded a similar result, but with the
constants depending on $\sigma_{0}$.
\end{rem}

\begin{prop}
\label{prop:sparcity H2} Let $\mathcal{F}=(\mathcal{F}_{\delta})_{0<\delta\leq1}$
be a system of random curves with short variable cutoff in $\Lambda\subseteq\rr^{d}$
satisfying hypothesis \textbf{H2 }with parameters $\rho,\sigma$.
Let $\gamma>\max(4d,\sigma^{2})$ be a scaling factor satisfying the
inequality $\gamma^{4d}e^{K_{1}-K_{0}\varepsilon\sqrt{\gamma}}<\frac{1}{8}$,
with $K_{0},K_{1}$ being the constants from the above lemma. Then
there exist constants $C,c>0$, depending only on $\sigma$ and $\Lambda$,
such that for all integer $k_{0}$,
\[
\pp_{\delta}(\textup{straight runs are }(\gamma,k_{0})\textup{-sparse down to scale \ensuremath{\delta} for all curves in \ensuremath{\mathcal{F}_{\delta}})}\geq1-Ce^{-ck_{0}}.
\]
\end{prop}

\begin{proof}
Combining Lemma \ref{lem:prob bound on existence of straigh runs}
and the definition of sparsity of straight runs, we see that:

\begin{multline}
\pp_{\delta}(\textup{straight runs are not }(\gamma,k_{0})\textup{-sparse down to scale \ensuremath{\delta})}\\
\leq\sum_{n=\lceil k_{0}/2\rceil}^{\infty}\sum_{1\le k_{1}<\cdots<k_{n}\le2n,\,\gamma^{-k_{n}}\ge\delta}\pp_{\delta}\left(\textup{\ensuremath{\begin{array}{c}
\textup{there is a nested sequence of straight runs}\\
\textup{at scales}\gamma^{-k_{1}},\gamma^{-k_{2}},...,\gamma^{-k_{n}}
\end{array}}}\right)\\
\leq\sum_{n=\lceil k_{0}/2\rceil}^{\infty}\sum_{1\le k_{1}<\cdots<k_{n}\le2n}K_{1}\gamma^{4dn}e^{K_{1}n-K_{0}\sqrt{\gamma}\varepsilon n}\\
\le K\sum_{n=\lceil k_{0}/2\rceil}^{\infty}{2n \choose n}\gamma^{4dn}e^{K_{1}n-K_{0}\sqrt{\gamma}\varepsilon n}\\
\leq K\sum_{n=\lceil k_{0}/2\rceil}^{\infty}{2n \choose n}\left(\frac{1}{8}\right)^{n}\leq K\sum_{n=\lceil k_{0}/2\rceil}^{\infty}4^{n}\left(\frac{1}{8}\right)^{n}\leq Ce^{-ck_{0}},\label{eq:really long ineqaulity}
\end{multline}
which is our desired result.
\end{proof}
We summarize the results of this section in the following proposition:
\begin{prop}
\label{prop:capestimate aizenmann} In a system satisfying the conditions
of Theorem \ref{thm:exact roughness},there exists a random variable
$k_{0;\delta}$ such that for any single curve $\mathcal{C}\in\mathcal{F}_{\delta}$:
\begin{equation}
\capi_{s;\delta}(\mathcal{C})\geq\left(\left(\frac{\gamma}{m}-1\right)\cdot\diam(\mathcal{C})\right)^{s}\left(\gamma^{s\cdot k_{0;\delta}}+\frac{\beta}{1-\gamma^{s}/\beta}\right)^{-1},\label{eq:capestimate}
\end{equation}
where $k_{0;\delta}=\min(k_{0}\geq0\textup{ : straight runs are \ensuremath{k_{0;\delta}} sparse down to scale \ensuremath{\delta})}$
is a random variable which satisfies:
\begin{equation}
\pp_{\delta}(k_{0;\delta}>N)\leq Ce^{-cN},N\in\nn,\label{eq:k0 prob estimate}
\end{equation}
$\gamma>\max(4d,\sigma^{2})$ is a scaling factor such that $\gamma^{4d}e^{K_{1}-K_{0}\varepsilon\sqrt{\gamma}}<\frac{1}{8}$,
$m$ is an integer in $[\gamma/2,\gamma]$, $\beta=\sqrt{m(m+1)}$,
and $\gamma^{s}<\beta$.
\end{prop}

Before moving on, note that we can simplify the required inequality
on $\gamma$ and rewrite it instead as:
\begin{equation}
\gamma>K\varepsilon^{-2}\log^{2}\left(\varepsilon\right)=:\gamma_{0}\label{eq:gamma inequality}
\end{equation}
where $K$ is a constant independent of $\rho$ (but depending on
the dimension and $\sigma$), and recalling that $\rho=1-\varepsilon$,
where $0<\varepsilon<\frac{1}{10}$.

\subsection{Completing the Argument}

\textbf{Proof of Theorem \ref{thm:exact roughness}: }Pick a value
$\gamma\in(\gamma_{0},2\gamma_{0})$ as in (\ref{eq:gamma inequality}),
with the additional property that $\gamma-\frac{1}{4}$ is an integer.
Denote $m=\left\lfloor \gamma\right\rfloor =\gamma-\frac{1}{4}$,
$\beta=\sqrt{m(m+1)}$ as before and take $s$ for which
\begin{equation}
\gamma^{s}=m\left(1+\frac{1}{m}\right)^{3/8}\leq m\left(1+\frac{1}{m}\right)^{1/2}=\beta.\label{eq:gamma bound}
\end{equation}
Then by Proposition \ref{prop:capestimate aizenmann}, for any curve
$\mathcal{C}$ in $\mathcal{F_{\delta}}$ such that $\diam(\mathcal{C})\geq1$:
\begin{equation}
\capi_{s;\delta}(\mathcal{C})\geq(\frac{\gamma}{m}-1)^{s}\left(\gamma^{s\cdot k_{0;\delta}}+\frac{\beta}{1-\gamma^{s}/\beta}\right)^{-1}=\frac{1}{4m}\left(\gamma^{s\cdot k_{0;\delta}}+\frac{\beta}{1-\gamma^{s}/\beta}\right)^{-1}.\label{eq:bettercapestimate}
\end{equation}
Additionally, the value of $s$ we chose also has the property that:
\begin{equation}
\frac{\gamma^{s}}{\beta}=\left(1+\frac{1}{m}\right)^{-1/8}=\left(\frac{m}{\beta}\right)^{1/4},\label{eq:gamma equality}
\end{equation}
so by (\ref{eq:capestimate}) and (\ref{eq:k0 prob estimate}) we
get for this value of $s$ and any $\mu>0$ small enough:
\begin{multline}
\pp(T_{s,1;\delta}<\mu)\leq\pp\left(\frac{1}{4m}\left(\gamma^{s\cdot k_{0;\delta}}+\frac{\beta}{1-\gamma^{s}/\beta}\right)^{-1}<\mu\right)=\\
=\pp\left(\gamma^{s\cdot k_{0;\delta}}+\frac{\beta}{1-\gamma^{s}/\beta}>\frac{1}{4\mu m}\right)\leq\pp\left(\gamma^{s\cdot k_{0;\delta}}+\frac{\beta}{1-(m/\beta)^{1/4}}>\frac{1}{4\mu\beta}\right)=\\
=\pp\left(\gamma^{s\cdot k_{0;\delta}}+\frac{\beta^{5/4}}{\beta^{1/4}-m^{1/4}}>\frac{1}{4\beta\mu}\right)\leq\pp\left(\gamma^{s\cdot k_{0;\delta}}+4\beta^{5/4}>\frac{1}{4\beta\mu}\right)=\\
=\pp\left(k_{0;\delta}>\frac{\log\left(\frac{1}{4\mu\beta}-4\beta^{5/4}\right)}{s\log(\gamma)}\right)\leq\exp\left(-c\frac{\log\left(\frac{1}{4\beta\mu}-4\beta^{5/4}\right)}{s\log(\gamma)}\right)\leq K(\beta\mu)^{\frac{K}{s\log(\gamma)}}\leq K\mu^{\frac{K}{s\log(\gamma)}},\label{eq:T prob bound}
\end{multline}
and in addition, for $\mu<1$: 
\begin{equation}
K\mu^{\frac{K}{s\log(\gamma)}}=K\mu^{\frac{K}{\log(\gamma^{s})}}\leq K\mu^{\frac{K}{\log(\beta)}}\leq K\mu^{\frac{K}{\log(\gamma)}}\leq K\mu^{\frac{K}{\log(1/\varepsilon)}},\label{eq:mu bound}
\end{equation}
proving (\ref{eq:gbound1}). Finally, in order to get (\ref{eq:alphabound1})
we see that:
\begin{multline}
s=\frac{\log\left(m\left(1+\frac{1}{m}\right)^{3/8}\right)}{\log(\gamma)}=\frac{\log\left(m\left(1+\frac{1}{m}\right)^{3/8}\right)}{\log(m+1/4)}=\\
\frac{\log\left(m\left(1+\frac{1}{m}\right)^{3/8}\right)}{\log\left(m(1+\frac{1}{4m})\right)}\geq\frac{\log\left(m(1+\frac{1}{m})^{3/8}\right)}{\log\left(m(1+\frac{1}{m})^{1/4}\right)}\geq1+\frac{1}{8}\frac{\log\left(1+\frac{1}{m}\right)}{\log\left(m(1+\frac{1}{m})^{1/4}\right)}\geq\\
\geq1+\frac{1}{16m\log(m)}\geq1+\frac{c}{\gamma\log(\gamma)}\geq1+c\frac{\varepsilon^{2}}{\ln(1/\varepsilon)^{3}}\label{eq:s bound}
\end{multline}
Therefore, for $\alpha\geq c\frac{\varepsilon^{2}}{\log(1/\varepsilon)^{3}}$
. In the same way, we also get an upper bound on $\alpha$.\qed

\section{\label{sec:Upper-Bound-in}Upper Bound in the Random-Field Ising
Model}

\subsection{Overview of the proof}

Our goal in this section will be to prove the first part of Theorem
\ref{thm:clength upper-1}. Namely:
\begin{thm}
\label{thm:clength upper} In the two-dimensional random-field Ising
model, for each fixed temperature $T\ge0$ and external field $\eta\in\rr$,
the correlation length satisfies
\begin{equation}
\zeta_{1}\leq\exp\left(\exp\left(O\left(\frac{J}{\varepsilon}\right)^{2}\right)\right),\label{eq:corr length bound upper}
\end{equation}
as $J/\varepsilon$ tends to infinity. 
\end{thm}

Throughout this section, we fix the temperature $T\ge0$ and external
field $\eta\in\rr$ and omit them from the notation.

The Aizenman--Harel--Peled \cite{expdecay} argument relies on analyzing
a disagreement percolation; see Section 3 there. At zero temperature,
the disagreement percolation is a random set of vertices obtained
as follows: one samples two independent instances of the Ising model
in $\Lambda(L)$ with the same realization of the random field, one
instance with $+$ boundary conditions and the other with $\lyxmathsym{\textendash}$
boundary conditions, and sets the disagreement vertices to be those
vertices where the two instances differ (at such vertices the instance
with $+$ boundary conditions lies strictly above the instance with
$-$ boundary conditions, as the temperature is zero). At positive
temperature, a more complicated construction is used and the resulting
disagreement percolation is a random set consisting of both vertices
and edges. 

The main result proved for the disagreement percolation is stated
in the following lemma taken from \cite[Lemma 5.1]{expdecay}, which
makes use of the Aizenman--Burchard theorem. The zero-temperature
version of the lemma was first proved by \cite{ding2019exponentialzerotemp}.
In the following lemma we shall use the following notation: the operator
$\left\langle \cdot\right\rangle ^{\partial\mathcal{A}_{1,2}(\ell),+\backslash-}$
denotes expectation over the disagreement percolation for a fixed
realization of the random field, and the operator $\ee$ denotes expectation
over the random field itself.
\begin{lem}
\label{lem:5.1 ron} Set $\mathcal{A}_{1,2}(\ell):=\Lambda(2\ell)\setminus\Lambda(\ell)$
to be the annulus of side length $2\ell$, and let $A_{\alpha,\ell}$
denote the event that the annulus $\mathcal{A}_{1,2}(\ell)$ is crossed
by a path of disagreement percolation of length at most $\ell^{1+\alpha}$
(the path alternates between edges and vertices, or consists only
of vertices if $T=0$, and connects the inner boundary to the outer
boundary). Then there exists an $\alpha_{0}>0$ for which:
\begin{equation}
\lim_{\ell\rightarrow\infty}\ee\left[\left\langle \oo_{A_{\alpha_{0},\ell}}\right\rangle ^{\partial\mathcal{A}_{1,2}(\ell),+\backslash-}\right]\rightarrow0.\label{eq:1}
\end{equation}
\end{lem}

As in \cite[Theorem 5.5]{expdecay}, we define $\alpha$ to be $\frac{\alpha_{0}}{2}$,
where $\alpha_{0}$ is any value for which (\ref{eq:1}) holds. In
addition to that, we define $\ell_{1}=\ell_{0}^{2}$ where $\ell_{0}>0$
is the minimal value of $\ell$ for which the expression inside the
limit (\ref{eq:1}) is less then some universal constant $\kappa<1$
independent of $J,\varepsilon$ \cite[Theorem 5.5, see (5.47)]{expdecay}.
Then, as stated in \cite[(6.48)]{expdecay}, an upper bound for the
correlation length $\zeta_{1}$ is given by
\begin{equation}
C\max\left(2,\frac{J}{\varepsilon},\frac{1}{\alpha},\ell_{1}\right)^{\frac{C}{\alpha^{2}}},\label{eq:2}
\end{equation}
with $C$ being a universal constant. Our goal then, is to give estimates
for $\ell_{1}$ and $\alpha$ in terms of $J/\varepsilon$. The main
technical results of this section are the following bounds.
\begin{thm}
(Upper bound for correlation length) \label{thm:alpha and l bounds}The
values of $\alpha$ and $\ell_{1}$ satisfy:
\begin{equation}
\alpha=\exp\left[-O\left((\frac{J}{\varepsilon})^{2}\right)\right],\label{eq:3}
\end{equation}
\begin{equation}
\ell_{1}=\exp\left(\exp\left(O\left(\frac{J}{\varepsilon}\right)^{2}\right)\right).\label{eq:4}
\end{equation}
Consequently, due to (\ref{eq:2}), the correlation length $\zeta_{1}$
satisfies the upper bound $\zeta_{1}=\exp\left(\exp\left(O\left((\frac{J}{\varepsilon})^{2}\right)\right)\right)$,
proving Theorem \ref{thm:clength upper}.
\end{thm}

\subsection{Technical arguments regarding the disagreement percolation}

The proof of Lemma \ref{lem:5.1 ron} relies on regarding the disagreement
percolation as a system of random curves to which the Aizenman--Burchard
framework may be applied, and showing that this system satisfies hypothesis
\textbf{H2} (see Section \ref{subsec:Definitions}).

We may consider the disagreement percolation as a system of random
curves with short variable cutoff in the following way: For the disagreement
percolation in the discrete annulus $\mathcal{A}_{1\text{,2}}(\ell)$,
we first dilate this annulus by a factor proportional to $1/\ell$
to make it contained inside $[-8,8]^{2}\setminus(-4,4)^{2}$. The
system of random curves is defined to be the collection of all paths
in the rescaled disagreement percolation, alternating between edges
and vertices, and embedding these paths in $\rr^{2}$ in the natural
way. In addition, we also only consider the intersection of the curves
with the sub annulus $[-7,7]^{2}\backslash(-5,5)^{2}$. Overall this
allows us to define a system of random curves with short variable
cutoff $(\mathcal{F}_{\delta})_{0<\delta\le1}$ in $[-7,7]^{2}\backslash(-5,5)^{2}$,
by sampling for a given $\delta$ the disagreement curves in the measure
$\left\langle \cdot\right\rangle ^{\partial\mathcal{A}_{1\text{,2}}(\ell),+\backslash-}$
for $\ell$ proportional to $1/\delta$ and re-scaling as earlier.
Later, we will denote the obtained measure on this system of random
curves by $\pp_{\delta}$ (this measure is obtained by averaging over
both $\left\langle \cdot\right\rangle ^{\partial\mathcal{A}_{1\text{,2}}(\ell),+\backslash-}$
and the random field).

The following was proven in \cite[Theorem 5.2]{expdecay}
\begin{lem}
\label{lem:ron 5.2} Let $\mathcal{R}$ be a collection of rectangles
contained in the annulus $[-\frac{7}{4}\ell,\frac{7}{4}\ell]\backslash[-\frac{5}{4}\ell,\frac{5}{4}\ell]$
with the following properties:
\begin{enumerate}
\item The side lengths of each $R\in\mathcal{R}$ are $\ell(R)\times5\ell(R)$
with $\ell(R)\in[10,\frac{1}{160}\ell]$.
\item The $\ell_{1}(\rr^{2})$ distance between distinct $R_{1},R_{2}\in\mathcal{R}$
is at least $60\max(\ell(R_{1}),\ell(R_{2}))$.
\end{enumerate}
Let $\mathcal{D}(\mathcal{R})$ be the event that all the rectangles
in $\mathcal{R}$ are crossed by the disagreement percolation (after
embedding the disagreement percolation in $\rr^{2}$ in the natural
way). Then there exist universal constants $c,C>0$ for which:
\begin{equation}
\ee\left[\left\langle \oo_{\mathcal{D}(\mathcal{R})}\right\rangle ^{\partial\mathcal{A}_{1\text{,2}}(\ell),+\backslash-}\right]\leq\left(1-c\exp\left(-C\left(\frac{J}{\varepsilon}\right)^{2}\right)\right)^{\left|\mathcal{R}\right|}.\label{eq:lemma 5 2 expectation bound}
\end{equation}
\end{lem}

Unfortunately, this weaker notion of well separatedness, alongside
the requirement of only having rectangles with length greater than
10 does not give us hypothesis \textbf{H2 }quite yet. What follows
is a rather technical argument on how we can convert this into a system
of random curves satisfying hypothesis \textbf{H2 }with parameters:
\begin{equation}
\sigma=200,\:\rho=1-c\exp\left(-C\left(\frac{J}{\varepsilon}\right)^{2}\right).\label{eq:h2 parameters}
\end{equation}
Let $\mathcal{R}$ be a collection of well-separated rectangles in
$[-7,7]^{2}\backslash(-5,5)^{2}$ of aspect ratio $\sigma=200$ and
length $>\delta$. For each rectangle $R\in\mathcal{R}$, we define
a rectangle $R'$ via the following: if the short side of $R$ is
of length $a$, we define $R'$ to be the $3a\times15a$ rectangle
with the same center as $R$, such that the sides of length $10a$
in $R'$ are parallel to those of length $a$ in $R$ . This way,
if $R$ is crossed by a curve, so is $R'$. Denote by $\mathcal{R}'$
the collection of all such sub-rectangles with the additional constraint
that their diameter is less than $1/1600$. Then for any pair $R_{1}',R_{2}'\in\mathcal{R}'$,
as the distance from $R_{1}$ to $R_{2}$ is at most $\max(\diam(R_{1}),\diam(R_{2}))$.
Denote $a_{1},a_{2}$ the lengths of the short sides of $R_{1},R_{2}$
respectively. Then the euclidian distance between $R_{1}',R_{2}'$
is at least:
\begin{multline}
\max(\diam(R_{1}),\diam(R_{2}))-2a_{1}-2a_{2}\geq200\max(a_{1},a_{2})-2a_{1}-2a_{2}\\
\geq194\max(a_{1},a_{2})\geq180\max(a_{1},a_{2})=60\max(\ell(R_{1}'),\ell(R_{2}').\label{eq:technical argument}
\end{multline}
and since $\ell_{1}$ distance is greater than euclidean distance,
we get that the collection $\mathcal{R}'$ satisfies the conditions
of Lemma \ref{lem:ron 5.2}, and so the probability of there being
a crossing of all rectangles in $\mathcal{R}'$ and in particular
of all rectangles in $\mathcal{R}$ is less than $\left(1-c\exp\left(-C\left(\frac{J}{\varepsilon}\right)^{2}\right)\right)^{\left|\mathcal{R}'\right|}$.
Finally, since the size of $\mathcal{R}'$ is only a constant less
than that of $\mathcal{R}$ (as there is a bounded finite number of
rectangles that can fail the diameter condition), we get that the
system of random curves $(\mathcal{F}_{\delta})$ we defined indeed
satisfies hypothesis \textbf{H2 }with parameters (\ref{eq:h2 parameters}).

\subsection{Proof of Theorem \ref{thm:alpha and l bounds}}

Applying Theorem \ref{thm:exact roughness} to the curves $(\mathcal{F_{\delta}})$
defined earlier, we get that for $\alpha=e^{-C(\frac{J}{\varepsilon})^{2}}$
and every $\mu,\delta>0$,
\begin{equation}
\pp_{\delta}(T_{1+2\alpha,1;\delta}<\mu)\leq\mu^{c\left(\frac{\varepsilon}{J}\right)^{2}},\label{eq:T variable reminder}
\end{equation}
where $T_{1+2\alpha,1;\delta}=\inf_{C\in\mathcal{F}_{\delta},\diam(C)\geq1}\capi_{1+2\alpha;\delta}(C)$.

Should the event $A_{\alpha,\ell}$ from Lemma \ref{lem:5.1 ron}
occur for this value of $\alpha$ and some integer $\ell>0$, we would
get a crossing of the annulus in the disagreement percolation which
has less than $\ell^{1+\alpha}$ steps, and so a curve in $\mathcal{F}_{\delta}$,
where $\delta:=c/\ell$, of length at most $C\ell^{\alpha}$ which
crosses $[-7,7]^{2}\backslash(-5,5)^{2}$. Denoting this crossing
curve by $\gamma$, note that $\gamma$ must have diameter greater
than 1, and that we may cover $\gamma$ by at most $\ell^{1+\alpha}$
disks of diameter $\delta$ (centered at each of the vertices in the
rescaled disagreement path). So by (\ref{eq:capprop2}) we obtain
that:
\begin{equation}
\ell^{1+\alpha}\geq\capi_{1+2\alpha;\delta}(\gamma)\cdot\left(C/\ell\right)^{-1-2\alpha}.\label{eq:capbound peniul}
\end{equation}
Therefore $\capi_{1+2\alpha;\delta}\leq C\ell^{-\alpha}$, and in
particular $T_{1+2\alpha,1;\delta}\leq C\ell^{-\alpha}$. We conclude
by (\ref{eq:T variable reminder})
\begin{multline}
\ee\left[\left\langle \oo_{A_{\alpha,\ell}}\right\rangle ^{\partial\mathcal{A}_{1\text{,2}}(\ell),+\backslash-}\right]\leq\pp_{c/\ell}(T_{1+2\alpha,1;c/\ell}\leq C\ell^{-\alpha})\\
\leq\left(C\ell^{-\alpha}\right)^{c\left(\frac{\varepsilon}{J}\right)^{2}}\leq e^{-c\log(\ell)\alpha(\frac{\varepsilon}{J})^{2}}.\label{eq:bound on prob of crossing}
\end{multline}
The last expression is less than a universal constant when $\ell\geq e^{e^{C(\frac{J}{\varepsilon})^{2}}}$.
Recalling that we seek $\ell_{1}=\ell_{0}^{2}$, where $\ell_{0}$
is the minimal value for which $\ee\left[\left\langle \oo_{A_{\alpha,\ell}}\right\rangle ^{\partial\mathcal{A}_{1\text{,2}}(\ell),+\backslash-}\right]\leq\kappa<1$
we conclude that $\ell_{1}=\exp\left(\exp\left(O\left(\frac{J}{\varepsilon}\right)^{2}\right)\right))$,
which combined with (\ref{eq:2}) gives us our desired bound on the
correlation length.\qed

\section{\label{sec:Lower-Bound}Lower Bound}

\subsection{A Brief Overview}

We will now give a lower bound for the correlation length in the following
sense: 
\begin{thm}
\label{thm:lowerbound}In the \textbf{zero-temperature} 2D RFIM Ising
model with sufficiently small field strength $\varepsilon$ ,coupling
constant $J>0$, and inside the box $\Lambda(L)$ centered at 0 with
positive boundary conditions
\[
H(\sigma)=-J\sum_{u\sim v}\sigma_{u}\sigma_{v}+\varepsilon\sum_{v}h_{v}\sigma_{v}.
\]
Where the random fields $h_{i}$ are independent random variables
satisfying $\ee[h_{i}]=0$ and a sub-Gaussian bound $\pp(\left|h_{i}\right|>t)<e^{-\half t^{2}}$.
Then for all $0<\delta<\frac{1}{2}$ there exists a constant $C=C(\delta)>0$
independent of $\varepsilon,J,L$ such whenever $L<e^{C\left(\frac{J}{\varepsilon}\right)^{\frac{2}{3}}}$then
\[
\pp(\forall i,\sigma_{i}^{+}=1)>1-\delta.
\]
Where $\sigma_{i}^{+}$ denotes the spin at $i$ in the ground state.
In particular, we get a lower bound $\zeta_{2}\geq e^{C\left(\frac{J}{\varepsilon}\right)^{\frac{2}{3}}}$
on the correlation length, at zero temperature.
\end{thm}

Note that we do not require the random variables $h_{i}$ to be identically
distributed. This gives us a slightly more generalized result. We
can easily work around this limitation by using the following Hoeffding
type inequality \cite[Theorem 2.2.6]{vershynin2018high}
\begin{lem}
\label{lem:gaussianhoefsding}Let $X_{1},X_{2},..$ be a collection
of independent random variables with $\ee[X_{i}]=0$ for all $i$.
Suppose there exists a constant $\psi$ such that $\ee[e^{X_{i}^{2}/\psi^{2}}]\leq2$
for all $i$. Then there exists a universal constant $c$ such that:
\[
\pp\left(\left|\sum_{i=1}^{n}X_{i}\right|\geq t\right)\leq\exp\left(-\frac{ct}{n\psi^{2}}\right).
\]
\end{lem}

In order to prove the theorem, we give a modified version of the proof
sketch presented by Fisher, Fr\"{o}hlich and Spencer\cite{fisher}. They
gave a proof of magnetization in the 3D RFIM, but under the assumption
what is called ``no contours within contours''. Essentially, they
assume an argument that is only known to be true in the case that
each $-$ component does not have any ``holes''. E.g, all components
of $-$ spins are simply connected.

Thankfully, since we only care about correlation length in 2D, as
it turns out we can circumvent this issue via the following steps.
\begin{enumerate}
\item First, we prove that for that sufficiently small box size, the probability
of seeing a \textbf{simply connected }set containing 0 with disagreements
in the boundary is low. This will be done using the methods developed
in \cite{fisher}.
\item Then, we extend this argument to simply connected sets with disagreements
in the boundary, but not necessarily ones containing the origin 0.
That is, we will show that up to length $e^{C\varepsilon^{-\frac{2}{3}}}$
we should expect with high probability that there will be \textbf{no}
simply connected sets with constant $+$ or $-$ signs with disagreements
on the boundary.
\item After that, we deduce that there must be no sites with a ``$-$''
configuration with high probability via the following deterministic
argument; Suppose there was a site with a ``$-$'' configuration,
then we may look at the maximal connected component containing this
site with all ``$-$'' configurations. As we show in 2., with a
high probability there are no simply connected constant sign maximal
components. In particular, our ``$-$'' component cannot be simply
connected, so it must contain a ``$+$'' component inside of it.
But that + component also cannot be simply connected, so it must contain
a ``$-$'' component within. We repeat to infinity, and reach an
obvious contradiction as we are in a finite discrete lattice.
\end{enumerate}
To formalize these steps, we first give a few definitions: 
\begin{defn}
\label{def:connected}A subset $A\subseteq\mathbb{Z}^{2}$ is called
\textbf{connected }if it is connected as a sub-graph of the integer
lattice. Furthermore, $A$ is called \textbf{simply connected }if
$A$ is connected and also $\mathbb{Z}^{2}\backslash A$ is connected.
We call $\left|A\right|$ the \textbf{area }of $A$.
\end{defn}

\begin{defn}
\label{def:boundry}The boundary of a subset $A\subseteq\mathbb{Z}^{2}$,
$\partial A$, is the collection of all edges in $\mathbb{Z}^{2}$
with one end in $A$ and the other in $\mathbb{Z}^{2}\backslash A$.
We call $\left|\partial A\right|$ the \textbf{perimeter }of $A$,
or the \textbf{length }of $\partial A$.
\end{defn}

\begin{defn}
\label{def:component}For a given configuration $\sigma:\Lambda(L)\rightarrow\{-1,1\}$,
we say a subset $A\subseteq\Lambda(L)$ is a \textbf{connected component
}of $\sigma$ if $A$ is connected, $\sigma$ is constant on $A$,
and there is no $B\supset A$ such that $B$ is connected and $\sigma$
is constant on $B$ (e.g, there are disagreements on the boundary
of $A$).
\end{defn}

\subsection{A Random Field Argument}

We first rephrase our problem to be one of the random field instead
of the Ising model. Suppose that in a ground state for a given random
field $(h_{i})$ then the origin 0 is contained in a simply connected
component $\Gamma$. Then we should expect the field inside $\Gamma$
to be bigger than the length of the boundary of $\Gamma$ in absolute
value. Indeed, if:
\begin{equation}
\left|\varepsilon\sum_{i\in\Gamma}h_{i}\right|<\frac{J}{2}\left|\partial\Gamma\right|,\label{eq:field is bigger 1}
\end{equation}
then we may reduce the energy of the configuration by changing all
the signs inside $\Gamma$ to match the signs next to the boundary,
a contradiction to $\Gamma$ being in the ground state.

Thus, we deduce that if 0 is contained in such simply connected component
$\Gamma$, it must suffice that:
\begin{equation}
\left|\sum_{i\in\Gamma}h_{i}\right|\geq\frac{J}{2\varepsilon}\left|\partial\Gamma\right|.\label{eq:field is bigger 2}
\end{equation}
So we will prove the following:
\begin{thm}
\label{thm:lowerboundRF}For any $\delta>0$ there exists a constant
$c>0$ such that for any $L<\exp\left(c(\frac{J}{\varepsilon})^{\frac{2}{3}}\right)$,
then:\\
\begin{equation}
\pp(\textup{There exists a simply connected set \ensuremath{\Gamma}\ensuremath{\ensuremath{\subseteq\Lambda}(L)} such that \ensuremath{\left|\sum_{i\in\Gamma}h_{i}\right|\geq\frac{J}{2\varepsilon}\left|\partial\Gamma\right|)<\delta}}.\label{eq:maintheoremRF}
\end{equation}
\end{thm}

From simple geometry we know that $\left|\partial\Gamma\right|\geq\fourth\sqrt{\left|\Gamma\right|}$,
so applying Lemma \ref{lem:gaussianhoefsding} with $\psi=2$ gives
us that for a given simply connected component $\Gamma\ni0$ then:
\begin{equation}
\pp\left(\left|\sum_{i\in\Gamma}h_{i}\right|\geq\frac{J}{2\varepsilon}\left|\partial\Gamma\right|\right)\leq e^{-c\left(\frac{J}{\varepsilon}\right)^{2}},\label{eq:probability that the field is bigger}
\end{equation}
for some universal constant $c$. Unfortunately however, the number
of such components is exponential in the box size $L$, so we cannot
just use a union bound on all such components. To get around this
problem, we will abuse the fact that there are many such components
very similar to each other. That is the motivation behind the so-called
``coarse graining'' presented in \cite{fisher}.
\begin{rem*}
As in the previous sections, we will use $\const$ to denote some
positive universal constant which may increase from line to line.
\end{rem*}

\subsection{The Coarse Graining}
\begin{defn}
\label{def:square cover}We call the tiling of $\mathbb{Z}^{2}$ by
disjoint $2^{k}\times2^{k}$ squares $\mathcal{T}_{k}$ $(k\geq0)$.
That is:

$\mathcal{F}_{k}$=$\left\{ \{m2^{k},m2^{k}+1,..,m2^{k}+2^{k}-1\right\} \times\{n2^{k},n2^{k}+1,..,n2^{k}+2^{k}-1\}:n,m\in\mathbb{Z}\}$.
\end{defn}

\begin{defn}
\label{def:Admissable}For a subset $\Gamma\subseteq\mathbb{Z}^{2}$,
we say that a square $s\in\mathcal{T}_{k}$ is \textbf{admissible
}with respect to $\Gamma$ if $\left|s\cap\Gamma\right|\geq2^{2k-1}$.
That is, the majority of vertices in $s$ are also in $\Gamma$.
\end{defn}

We may now define the coarse gaining with respect to a starting set
$\Gamma$:
\begin{defn}
\label{def:coarse graining}For a (simply connected) set $0\in\Gamma\subseteq\mathbb{\mathbb{Z}}^{2}$,
the \textbf{coarse graining }of $\Gamma$ to be a sequence of sets
$\Gamma_{0},\Gamma_{1},\Gamma_{2},...\subseteq\mathbb{Z}^{2}$ as
follows: For each $k\geq0$, define $\Gamma_{k}$ to be the union
of all admissible $2^{k}\times2^{k}$ squares $c\in\mathcal{T}_{k}$
with respect to $\Gamma$. Note that under this definition, $\Gamma_{0}=\Gamma$.
\begin{figure}
\begin{centering}
\includegraphics[scale=0.55]{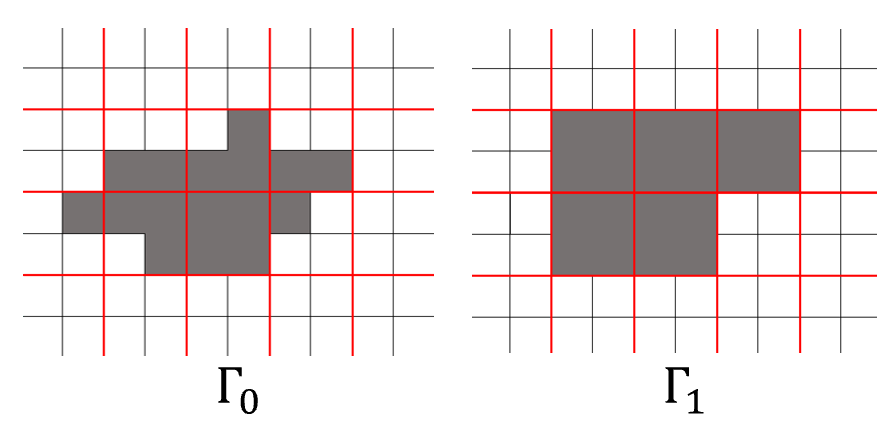}
\par\end{centering}
\caption{An example of one step of the coarse graining. If a square in the
red grid has two or more gray squares, it becomes gray in $\Gamma_{1}$,
otherwise it becomes white.}
\end{figure}
\end{defn}

The coarse graining will be useful as it will allow us to look at
the event that the field is large inside the coarse graining of a
starting set, instead of the original set. This will be very useful
as there will be far fewer coarse grained sets than simply connected
starting sets. For this, we need two key lemmas:
\begin{prop}
\textbf{(Key Lemma 1)\label{lem:lemma 1}} For any starting set $\Gamma$,
we have that for all $k>1$:
\begin{enumerate}
\item $\left|\partial\Gamma_{k}\right|<8\left|\partial\Gamma\right|$.
\item $\left|\Gamma_{k}\backslash\Gamma_{k-1}\right|,\left|\Gamma_{k-1}\backslash\Gamma_{k}\right|<16\cdot2^{k}\left|\partial\Gamma\right|$
\item $\left|\Gamma_{k}\right|<32\cdot2^{k}\left|\partial\Gamma\right|+\left|\Gamma\right|$.
\end{enumerate}
\end{prop}

\begin{prop}
\textbf{(Key Lemma 2)}\label{lem:lemma 2} The number of possible
coarse grained sets $\Gamma_{k}$ for a starting simply connected
set $\Gamma\subseteq\Lambda(L)$ is less than:
\begin{equation}
L^{2}\cdot\exp\left(\const\cdot\frac{\left|\partial\Gamma\right|k+\left|\partial\Gamma\right|\log\left|\partial\Gamma\right|}{2^{k}}\right).\label{eq:lemma2bound}
\end{equation}
\end{prop}

We delay the proof of the second key lemma to the end of the paper.
For now, let us prove the first key lemma:

\subsection*{Proof of Key Lemma 1:}

In this proof, $\const$ will denote some constant independent of
the starting set $\Gamma$ and also independent of $k$.

We emulate the methods in the proof of proposition 1 in \cite{fisher}
for 2D. To prove (1), note that $\left|\partial\Gamma_{k}\right|$
is exactly $2^{k}$ times the number of distinct unordered pairs of
adjacent squares $s_{1},s_{2}\in\mathcal{T}_{k}$ where $s_{1}$ is
admissible and $s_{2}$ is not. That is, $\left|\Gamma\cap s_{1}\right|\geq2^{2k-1}$
and $\left|\Gamma\cap s_{2}\right|<2^{2k-1}$. Denoting $s=s_{1}\cup s_{2}$,
we get that $2^{2k-1}\leq\left|\Gamma\cap s\right|<3\cdot2^{2k-1}$.
We want to give an estimation to the length of $\left|\partial\Gamma\cap\textup{int}(s)\right|$,
where the interior is again in the Euclidean metric. That is, we want
to give an estimation of the number of length 1 edges of $\partial\Gamma$
which are strictly contained inside $s$.

Indeed, project $\Gamma\cap s_{1}$ onto the shared side of $s_{1}$
and $s_{2}$. Let $q$ denote the number of edges projected. Since
$\left|\Gamma\cap s_{1}\right|\geq2^{2k-1}$, we must have $q\geq2^{k-1}$.
Subdivide $s_{2}$ into $2^{k}$ columns of length $2^{k}$, perpendicular
to the shared side of $s_{1}$ and $s_{2}$. For one of the $q$ elements
projected, it will contribute an element of $\partial\Gamma\cap\textup{int}(s)$
if the column touching it has at least one element not in $\Gamma$.
If every column has at least one such element, we get $\left|\partial\Lambda\cap\textup{int}(s_{1}\cup s_{2})\right|\geq2^{k-1}$.
Otherwise, there is a column fully contained in $\Gamma$, denote
it by $\mathcal{R}$. We can again subdivide $s_{2}$ to $2^{k}$
``rows'' of length $2^{k}$, this time perpendicular to $\mathcal{R}$.
Each element in $\mathcal{R}$ will contribute at least one element
to $\partial\Lambda\cap\textup{int}(s_{1}\cup s_{2})$ if its corresponding
row has an element not in $\Gamma$. Since $\left|\Gamma\cap c_{2}\right|<2^{2k-1}$,
we know that there must be at least $2^{k-1}$ such rows. So again
we get $\left|\partial\Lambda\cap\textup{int}(s_{1}\cup s_{2})\right|\geq2^{k-1}$

Finally, go over all such pairs of $s_{1}$ and $s_{2}$ and count
the number of edges in $\left|\partial\Lambda\cap\textup{int}(s_{1}\cup s_{2})\right|$.
As each edge of $\partial\Gamma$ will be counted at most 4 times,
we get:
\begin{equation}
4\left|\partial\Gamma\right|\geq\sum_{s_{1},s_{2}}\left|\partial\Lambda\cap\textup{int}(s_{1}\cup s_{2})\right|\geq2^{k-1}\cdot\left[\#(s_{1},s_{2})\right]=2^{k-1}\cdot\frac{\left|\partial\Gamma_{k}\right|}{2^{k}}=\frac{1}{2}\left|\partial\Gamma_{k}\right|.\label{eq:lemma1proofeq1}
\end{equation}
Where the sum is over unordered pairs of adjacent squares $s_{1},s_{2}$
where one is admissible and the other is not, giving us our desired
result.

Now to prove part 2. It is enough to show that $\left|\Gamma_{k}\setminus\Gamma_{k-1}\right|,\left|\Gamma_{k-1}\setminus\Gamma_{k}\right|<\const\cdot2^{k}\left|\partial\Gamma\right|$.
For the first inequality, suppose $s\in\mathcal{T}_{k}$ is a square
such that $s\cap(\Gamma_{k}\backslash\Gamma_{k-1})\neq\emptyset$.
We can write $s$ as a union of four disjoint squares in $\mathcal{T}_{k-1}$.
Then at least one of these squares is not admissible, else we have
$s\subseteq\Gamma_{k-1}$. But $s$ itself must be admissible, as
$s\cap\Gamma_{k}\neq\emptyset$, so one of these four squares must
be admissible. We conclude that $s$ contains two adjacent squares
$s_{1},s_{2}\in\mathcal{T}_{k-1}$ such that $s_{1}$ is admissible
and $s_{2}$ is not. So, by the proof of part 1 of the lemma, we get
that the number of such squares $s\in\mathcal{T}_{k}$ such that $s\cap(\Gamma_{k}\backslash\Gamma_{k-1})\neq\emptyset$
is at most $\frac{1}{2^{k-1}}\left|\partial\Gamma_{k-1}\right|\leq\frac{16}{2^{k}}\left|\partial\Gamma\right|$.
Finally, we note that:

\begin{multline}
\left|\Gamma_{k}\backslash\Gamma_{k-1}\right|\leq(2^{k})^{2}\cdot(\textup{number of \ensuremath{s\in\eff_{k}}with \ensuremath{s\cap(\Gamma_{k}\backslash\Gamma_{k-1})\neq\emptyset})}<\\
<2^{2k}\cdot16\cdot2^{-k}\left|\partial\Gamma\right|=16\cdot2^{k}\left|\partial\Gamma\right|.\label{eq:lemma1proofeq2}
\end{multline}
As for $\left|\Gamma_{k-1}\backslash\Gamma_{k}\right|$, the proof
is nearly identical, as for any $s\in\mathcal{T}_{k}$ such that $s\cap(\Gamma_{k-1}\backslash\Gamma_{k})\neq\emptyset$
must too contain an admissible square in $\mathcal{T}_{k-1}$ and
an inadmissible one. Then we continue exactly as before to obtain
$\left|\Gamma_{k-1}\backslash\Gamma_{k}\right|<16\cdot2^{k}\left|\partial\Gamma\right|$,
as we wanted.

For the final part, note that $\Gamma_{k}\subseteq\bigcup_{i=2}^{k}\Gamma_{i}\backslash\Gamma_{i-1}\cup\Gamma$,
and so by \ref{eq:lemma1proofeq2}:
\begin{equation}
\left|\Gamma_{k}\right|=\sum_{i=2}^{k}\left|\Gamma_{i}\backslash\Gamma_{i-1}\right|+\left|\Gamma\right|\leq\sum_{i=2}^{k}16\cdot2^{i}\left|\partial\Gamma\right|+\left|\Gamma\right|=32\cdot2^{k}\left|\partial\Gamma\right|+\left|\Gamma\right|.\label{eq:end of lemma}
\end{equation}
\qed
\begin{rem*}
Before we continue, note that the definition of the coarse graining
is for general subsets of $\mathbb{Z}^{2}$, not just subsets of $\Lambda(L)$.
As we will be confined to $\Lambda(L)$, we shall also confine the
coarse graining to remain inside of it by intersecting the coarse
grained sets $\Gamma_{k}$ with $\Lambda(L)$. It is easy to check
that the two key lemmas still hold in this case.
\end{rem*}
\begin{rem*}
Note that the proof did not require the starting set $\Gamma$ itself
to be simply connected. Nevertheless all our applications of the lemma
will involve a simply connected $\Gamma$. In which case, we may 'improve'
the third part of the lemma to $\left|\Gamma_{k}\right|<32\cdot2^{k}\left|\partial\Gamma\right|+\left|\Gamma\right|\leq32\cdot2^{k}\left|\partial\Gamma\right|+\left|\partial\Gamma\right|^{2}$.
This modification will be useful for us later.
\end{rem*}

\subsection{Back to the Random Field}

Returning to the random field $(h_{i})$ in $\Lambda(L)$, we define
the following three events:
\begin{defn}
\label{def:sigmaL}For an integer $0<\ell<4L^{2}$, define the ``low
field in corridor'' event:

$\mathcal{E}_{k}(\ell)=\left\{ \left|\sum_{i\in\Gamma_{k+1}\backslash\Gamma_{k}}h_{i}\right|,\left|\sum_{i\in\Gamma_{k}\backslash\Gamma_{k+1}}h_{i}\right|\leq c_{\ell},\text{for all simply connected \ensuremath{\Gamma} with \ensuremath{\left|\partial\Gamma\right|=\ell}}\right\} .$

Where $1\leq k<N_{\ell}$ is an integer, and $N_{\ell}=\left\lfloor \log_{2}(\ell)\right\rfloor $,
and $c_{\ell}=\frac{J\ell}{8N_{\ell}\varepsilon}.$ In addition define:

$\mathcal{E}_{N_{\ell}}(\ell)=\left\{ \left|\sum_{i\in\Gamma_{N_{\ell}}}h_{i}\right|\leq c_{\ell},\textup{for all starting simply connected sets \ensuremath{\Gamma} with \ensuremath{\left|\partial\Gamma\right|=\ell}}\right\} ,$

$Q=\left\{ \textup{There exists a simply connected set \ensuremath{\Gamma\ni0} such that \ensuremath{\left|\sum_{i\in\Gamma}h_{i}\right|\geq\frac{J}{2\varepsilon}\left|\partial\Lambda\right|}}\right\} .$
\end{defn}

Recalling theorem\ref{thm:lowerboundRF} our goal is to show the probability
of $Q$ is low.\\
It is easy to check that:
\begin{equation}
Q^{c}\supseteq\bigcap_{\ell=1}^{4L^{2}}\bigcap_{k=1}^{N_{\ell}}\mathcal{E}_{k}(\ell),\label{eq:Qinclusion}
\end{equation}
and so:
\begin{equation}
\pp(Q)\leq\sum_{\ell=1}^{4L^{2}}\sum_{k=1}^{N_{\ell}}\pp\left(\mathcal{E}_{k}(\ell)^{c}\right).\label{eq:boundofP(Q)}
\end{equation}
So in order to prove theorem \ref{thm:lowerboundRF} our goal will
be to estimate the probabilities of the low field corridor events
$\mathcal{E}_{k}(\ell)$, and then the sum \ref{eq:boundofP(Q)}.
\begin{prop}
\label{prop:sigma prob bound}For a given simply connected set $\Gamma$,
with $\left|\partial\Gamma\right|=\ell$ then:
\[
\pp\left(\left|\sum_{i\in\Gamma_{k+1}\backslash\Gamma_{k}}h_{i}\right|\geq c_{\ell}\textup{ or }\left|\sum_{i\in\Gamma_{k}\backslash\Gamma_{k+1}}h_{i}\right|\geq c_{\ell}\right)\leq\exp\left(-\left(\frac{J}{\varepsilon}\right)^{2}\frac{\const\cdot\ell}{\log(\ell)^{2}\cdot2^{k}}\right).
\]
In addition:
\[
\pp\left(\left|\sum_{i\in\Gamma_{N_{\ell}}}h_{i}\right|\geq c_{\ell}\right)\leq\exp\left(-\left(\frac{J}{\varepsilon}\right)^{2}\frac{\const\cdot\ell}{2^{N_{\ell}}\log(\ell)^{2}}\right).
\]
\end{prop}

\begin{proof}
By \ref{eq:probability that the field is bigger} and the second part
of the first key lemma, when $k<N_{\ell}$:
\begin{multline}
\pp\left(\left|\sum_{i\in\Gamma_{k+1}\backslash\Gamma_{k}}h_{i}\right|\geq c_{\ell}\right)\leq\exp\left(-c_{\ell}^{2}\cdot\frac{\const}{\left|\Gamma_{k+1}\backslash\Gamma_{k}\right|}\right)\leq\exp\left(-c_{\ell}^{2}\cdot\frac{\const}{2^{k}\cdot\ell}\right)\leq\\
\leq\exp\left(-\left(\frac{J}{\varepsilon}\right)^{2}\left(\frac{\ell}{\log(\ell)}\right)^{2}\frac{\const}{2^{k}\cdot\ell}\right)=\exp\left(-\left(\frac{J}{\varepsilon}\right)^{2}\frac{\const\cdot\ell}{\log(\ell)^{2}\cdot2^{k}}\right),\label{eq:sigma k lower than N}
\end{multline}
and the same holds for $\pp\left(\left|\sum_{i\in\Gamma_{k+1}\backslash\Gamma_{k}}h_{i}\right|\geq c_{\ell}\right)$,
hence the union bound gives us the first part of the proposition.

When $k=N_{\ell}$, by the third part of the first key lemma: 
\begin{multline}
\pp\left(\left|\sum_{i\in\Gamma_{N_{\ell}}}h_{i}\right|\geq c_{\ell}\right)\leq\exp\left(-c_{\ell}^{2}\cdot\frac{\const}{\left|\Gamma_{N_{\ell}}\right|}\right)\leq\exp\left(-c_{\ell}^{2}\cdot\frac{\const}{2^{N_{\ell}}\ell+\ell^{2}}\right)\leq\\
\leq\exp\left(-\left(\frac{J}{\varepsilon}\right)^{2}\left(\frac{\ell}{\log(\ell)}\right)^{2}\frac{\const}{2^{N_{\ell}}\ell+\ell^{2}}\right)\leq\exp\left(-\left(\frac{J}{\varepsilon}\right)^{2}\frac{\const\cdot\ell}{2^{N_{\ell}}\log(\ell)^{2}}\right).\label{eq:sigma k on Nl}
\end{multline}
\end{proof}
From here, we can find the probabilities of the $\mathcal{E}_{k}(\ell)$
events:
\begin{prop}
\label{prop:sigma event prob} For all $\ell,k\leq N_{\ell}$, then
\begin{equation}
\pp\left(\sige_{k}(\ell)^{c}\right)\leq L^{2}\exp\left(\frac{\ell\log(\ell)}{2^{k}}-\left(\frac{J}{\varepsilon}\right)^{2}\frac{\const\cdot\ell}{2^{k}\log(\ell)^{2}}\right).\label{eq:sigma k prob bound}
\end{equation}
\end{prop}

\begin{proof}
From the second key lemma, we know that the number of pairs $\Gamma_{k},\Gamma_{k-1}$
when $\Gamma$ is a starting set with $\left|\partial\Gamma\right|=\ell$
is at most $L^{2}\exp\left(\const\cdot\frac{\ell k+\ell\log(\ell)}{2^{k}}\right)$,
so by the union bound and proposition \ref{prop:sigma prob bound}
we get that whenever $k<N_{\ell}$
\begin{align}
\pp\left(\sige_{k}(\ell)^{c}\right) & \leq L^{2}\exp\left(\const\cdot\frac{\ell k+\ell\log(\ell)}{2^{k}}\right)\exp\left(-\left(\frac{J}{\varepsilon}\right)^{2}\frac{\const\cdot\ell}{\log(\ell)^{2}\cdot2^{k}}\right)\leq\label{eq:sigma k prob bound proof}\\
 & \leq L^{2}\exp\left(\const\frac{\ell\log(\ell)}{2^{k}}-\left(\frac{J}{\varepsilon}\right)^{2}\frac{\const\cdot\ell}{2^{k}\log(\ell)^{2}}\right).\nonumber 
\end{align}
For $k=N_{\ell}$, the proof is identical, as the number of possible
sets $\Gamma_{k}$ is still at most $L^{2}\exp\left(\const\cdot\frac{\ell\log(\ell)}{2^{k}}\right)$
.
\end{proof}
We are now ready to prove theorem \ref{thm:lowerboundRF}:

\subsubsection*{Proof of Theorem \ref{thm:lowerboundRF}:}

Reusing the event $Q$ from definition \ref{def:sigmaL}, we deduce
from proposition \nameref{prop:sigma event prob} that:
\begin{multline}
\pp(Q^{c})\leq\sum_{\ell=1}^{4L^{2}}\sum_{k=1}^{N_{\ell}}\pp\left(\mathcal{E}_{k}(\ell)^{c}\right)\leq\\
\leq\sum_{\ell=1}^{4L^{2}}\sum_{k=1}^{N_{\ell}}L^{2}\exp\left(\const\frac{\ell\log(\ell)}{2^{k}}-\left(\frac{J}{\varepsilon}\right)^{2}\frac{\const\cdot\ell}{2^{k}\log(\ell)^{2}}\right)\\
\leq\sum_{\ell=1}^{4L^{2}}\sum_{k=1}^{N_{\ell}}L^{2}\exp\left(\frac{\ell}{2^{k}}\left[\const\log(\ell)-\left(\frac{J}{\varepsilon}\right)^{2}\frac{\const}{\log(\ell)^{2}}\right]\right).\label{eq:prob Q bound 1}
\end{multline}
For appropriate constants, when $L\leq\exp\left(\const\left(\frac{J}{\varepsilon}\right)^{\frac{2}{3}}\right)$,
then $\log(\ell)\leq\const\left(\frac{J}{\varepsilon}\right)^{\frac{2}{3}}$
whenever $\ell\leq4L^{2}$ , and so $\const\log(\ell)-\left(\frac{J}{\varepsilon}\right)^{2}\frac{\const}{\log(\ell)^{2}}<0$.
Therefore:
\begin{multline}
\sum_{\ell=1}^{4L^{2}}\sum_{k=1}^{N_{\ell}}L^{2}\exp\left(\frac{\ell}{2^{k}}\left[\const\log(\ell)-\left(\frac{J}{\varepsilon}\right)^{2}\frac{\const}{\log(\ell)^{2}}\right]\right)\leq\\
\leq\sum_{\ell=1}^{4L^{2}}\sum_{k=1}^{N_{\ell}}L^{2}\exp\left(\left[\const\log(L)-\left(\frac{J}{\varepsilon}\right)^{2}\frac{\const}{\log(L)^{2}}\right]\right)\\
=4L^{4}\log(L)\exp\left(\left[\const\log(L)-\left(\frac{J}{\varepsilon}\right)^{2}\frac{\const}{\log(L)^{2}}\right]\right).\label{eq:prob Q bound 2}
\end{multline}
We conclude that for any $\delta>0$, we may find an appropriate constant
$C=C(\delta)$ such that when $L\leq\exp\left(\const\left(\frac{J}{\varepsilon}\right)^{\frac{2}{3}}\right)$,
then $\pp(Q)<\delta$.\qed

\subsection{Closing Arguments}

We are now ready to prove Theorem \ref{thm:lowerbound}:
\begin{proof}
Let $\delta>0$. From Theorem \ref{thm:lowerboundRF}, we know that
there exists a constant $c=c(\delta)>0$ such that whenever $L<\exp\left(c(\frac{J}{\varepsilon})^{\frac{2}{3}}\right)$,
then:
\begin{equation}
\pp\left(\textup{There exists a simply connected set \ensuremath{\Gamma\subseteq\Lambda(L)} on which \ensuremath{\left|\sum_{i\in\Gamma}h_{i}\right|\geq\frac{J}{2\varepsilon}\left|\partial\Gamma\right|}}\right)<\delta.\label{eq:prob is low for high field set}
\end{equation}
Suppose a configuration $\sigma:\Lambda(L)\rightarrow\{-1,+1\}$ is
not all $+1$. Then there must be a simply connected subset $\Gamma\subset\Lambda(L)$
with a constant sign for $\sigma$ and disagreements on the boundary.
Indeed, start at 0. If the component containing $v_{0}=0$ in $\sigma$
is simply connected, we are done. Otherwise, the component containing
$0$ has a hole. Pick any vertex $v_{1}$ in that hole. If the component
in $\sigma$ containing $v_{1}$ is simply connected, we are done.
Otherwise, that component has a hole and we can pick a vertex $v_{2}$
in it. We repeat until we reach a simply connected component $\Gamma$
with constant signs and disagreements on the boundary. The process
must end, otherwise we would get an infinite sequence of vertices
in the finite graph $\Lambda(L)$.

Finally, should the complement of the event in \ref{eq:prob is low for high field set}
hold, there cannot be a $-1$ in the ground state configuration $\sigma$.
Indeed, if there was a $-1$, by the above argument, there would be
a simply connected set $\Gamma$ with a constant sign and disagreements
on the boundary. Then we can flip the signs on $\Gamma$ to reduce
the energy by $J\left|\partial\Gamma\right|$, but add at most $\varepsilon\left|\sum_{i\in\Gamma}h_{i}\right|$
to the energy. As the complement of the even in \ref{eq:prob is low for high field set}
holds, we know that by doing this we must reduce the energy, as $J\left|\partial\Gamma\right|>\varepsilon\left|\sum_{i\in\Gamma}h_{i}\right|$.
This will give us a new configuration $\sigma'$ with lower energy,
a contradiction to $\sigma$ being the ground state.

In total, for $L<\exp\left(c(\frac{J}{\varepsilon})^{\frac{2}{3}}\right)$
it must hold with probability higher than $1-\delta$ that the ground
state will be the constant state $\sigma=+1$.
\end{proof}

\appendix

\section{\label{sec:Appendix:-Proving-the}Appendix: Proving the Second Key
Lemma}

In this appendix, we adapt the proof given in the Appendix of \cite{fisher}
to the case of two dimensions in order to prove Proposition \ref{lem:lemma 2}.

Let $\Gamma\subseteq\zz^{2}$ be a simply connected set containing
the origin $0$. Note that while the original set $\Gamma$ and its
boundary $\partial\Gamma$ are both connected, the coarse grained
sets $\Gamma_{k}$ and hence their boundaries $\partial\Gamma_{k}$
need not be connected. For example:
\begin{figure}
\centering{}\includegraphics[scale=0.45]{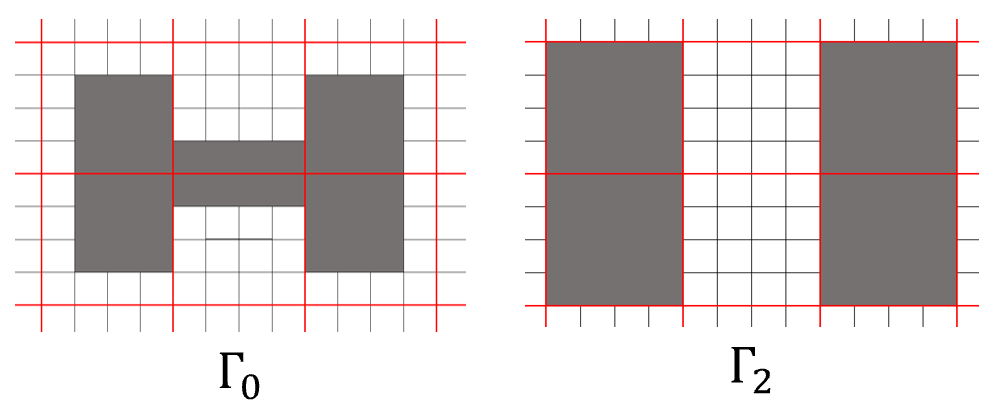}\caption{An example of a starting set $\Gamma_{0}$ for which the coarse grained
set $\Gamma_{2}$ is disconnected.}
\end{figure}

Nevertheless, $\partial\Gamma_{k}$ must have no more than $\frac{8\ell}{2^{k}}$
connected components when $\left|\partial\Gamma\right|=\ell$. Indeed,
each connected component of $\partial\Gamma_{k}$ must contain at
least two lines of length $2^{k}$. But from the first key lemma,
we know that $\left|\partial\Gamma_{k}\right|\leq8\left|\partial\Gamma\right|=8\ell$.
So the number of connected components of $\partial\Gamma_{k}$ must
indeed by no more than $\frac{4\ell}{2^{k}}$. We write $\partial\Gamma_{k}=\biguplus_{i}\gamma_{i}$,
where $\gamma_{i}$ are the connected components of $\partial\Gamma_{k}$.
Additionally denote $\alpha=\alpha(\ell):=\frac{4\ell}{2^{k}}$.

For a fixed set of points $x_{1},x_{2},..,x_{\alpha}\in\zz^{2}$,
and integer lengths $l_{1},..,l_{\alpha}\geq0$, $l_{1}+\cdots+l_{\alpha}\leq\frac{8\ell}{2^{k}}$,
define a function:
\begin{equation}
F_{\ell,k}(\{x_{i},l_{i}\}):=\left(\begin{array}{c}
\textup{Number of initial sets \ensuremath{0\in\Gamma} with \ensuremath{\left|\partial\Gamma\right|=\ell,}and for which \ensuremath{\left|\gamma_{i}\right|=2^{k}l_{i}},}\\
\textup{ and either \ensuremath{x_{i}}\ensuremath{\in\gamma_{i}} or \ensuremath{\gamma_{i}} is empty for each \ensuremath{i}}
\end{array}\right).\label{eq:F definition}
\end{equation}

Then the number of possible $\gamma_{i}$ with length $2^{k}l_{i}$
containing an arbitrary $x_{i}$ is at most $\exp(\const\cdot l_{i})$,
as $\gamma_{i}$ is comprised of $l_{i}$ segments of length $2^{k}$.
So as there are at most $\exp(\const\cdot l_{i})$ ways to pick those
segments. Therefore:
\begin{equation}
F_{\ell,k}(\{x_{i},l_{i}\})\leq\exp\left(\const\cdot\sum l_{i}\right)\leq\exp\left(\const\frac{\ell}{2^{k}}\right).\label{eq:F inequality}
\end{equation}

It now remains to bound the number of possible $l_{i}$'s and $x_{i}$'s.
First, for $l_{i}$, we know from basic combinatorics that the number
of ways to sum $\alpha$ positive integers to get a result less than
$\const\cdot\frac{\ell}{2^{k}}$ is at most $2^{\const\cdot\frac{\ell}{2^{k}}}.$
Now for the $x_{i}$'s. For them, note that from Proposition \ref{lem:lemma 1}
that if we want $F_{\ell,k}(\{x_{i},l_{i}\})\neq0$, we must have
that all the $x_{i}$'s are all contained in a set of diameter $\leq16\cdot\ell$
centered around the origin. Therefore, the number of $x_{i}$'s is
bounded by:
\begin{equation}
\binom{\const\cdot\ell^{2}}{\alpha}\leq\left(\const\cdot\frac{\ell^{2}}{\alpha}\right)^{\alpha}=\left(\const\cdot2^{k}\ell\right)^{\const\cdot\frac{\ell}{2^{k}}}\leq\exp\left(\const\cdot\frac{\ell\log(\ell)+\ell k}{2^{k}}\right).\label{eq:binom inequality}
\end{equation}
So at last, we get that the number of coarse grained sets $\Lambda_{k}$
starting from a set $0\in\Gamma\subseteq\zz^{2}$ with $\left|\partial\Gamma\right|=\ell$,
will be at most (by the above inequalities):
\begin{align}
\left(\textup{number of \ensuremath{l_{i}}'s}\right)\cdot\left(\textup{number of \ensuremath{x_{i}}'s}\right)\cdot\left(\textup{maximal value of \ensuremath{F_{\ell,k}}}\right) & \leq\label{eq:finalinequality}\\
\leq\exp\left(\const\frac{\ell}{2^{k}}\right)\cdot\exp\left(\const\cdot\frac{\ell\log(\ell)+\ell k}{2^{k}}\right).\exp\left(\const\cdot\frac{k\ell}{2^{k}}\right)\leq\exp\left(\const\cdot\frac{k\ell+\ell\log(\ell)}{2^{k}}\right).\nonumber 
\end{align}

To conclude the proof, we note that this bounds works for any starting
point in a set $x_{0}\in\Gamma$. So by going over all the $L^{2}$
in the $L\times L$ square, we get (\ref{eq:lemma2bound}).\qed

\bibliographystyle{plain}
\nocite{*}
\bibliography{ising}

\end{document}